\documentclass{amsart}

\usepackage{graphicx}         % For embedding figures and subfigures
\usepackage[tight]{subfigure} % For embedding figures and subfigures
\usepackage{url}              % For embedding URLs nicely in text

\newtheorem{theorem}{Assumptions}
\newtheorem{lemma}[theorem]{Lemma}
\newtheorem{corollary}[theorem]{Corollary}

\theoremstyle{definition}
\newtheorem{definition}[theorem]{Definition}
\newtheorem{example}[theorem]{Example}
\newtheorem{algorithm}[theorem]{Algorithm}

\newcommand{\deltaexp}{(\alpha-\alpha^{-1})z^{-1}}
\newcommand{\homfly}{HOMFLY-PT}

\newcommand{\link}{\mathcal{L}}
\newcommand{\poly}{\mathrm{poly}}

\newcommand{\R}{\mathbb{R}}
\newcommand{\regina}{\emph{Regina}}

\begin{document}

\title{The HOMFLY-PT polynomial is fixed-parameter tractable}
\author{Benjamin A.\ Burton}
\address{School of Mathematics and Physics \\
    The University of Queensland \\
    Brisbane QLD 4072 \\
    Australia}
\email{bab@maths.uq.edu.au}
\thanks{The author is supported by the Australian Research Council
    under the Discovery Projects funding scheme (DP150104108).}
% TODO: Fix MSC codes.
%\subjclass[2000]{%
%    Primary
%    57-04, % Manifolds & cell complexes / Machine computation & programs
%    57N10; % Topological manifolds / Topology of general 3-manifolds
%    Secondary
%    68W05, % Algorithms / non-numerical algorithms
%    57Q15} % Triangulating manifolds
% TODO: Fix keywords.
%\keywords{3-manifolds, algorithms, software, simplification,
%    normal surfaces, recognition, angle structures}

\begin{abstract}
    Many polynomial invariants of knots and links, including the
    Jones and HOMFLY-PT polynomials, are widely used in practice
    but \#P-hard to compute. It was shown by Makowsky in 2001 that
    computing the Jones polynomial is fixed-parameter tractable in the
    treewidth of the link diagram, but the parameterised complexity of the
    more powerful HOMFLY-PT polynomial remained an open problem.
    Here we show that computing HOMFLY-PT is fixed-parameter tractable
    in the treewidth, and we give the
    first sub-exponential time algorithm to compute it for arbitrary links.
\end{abstract}

\maketitle

%%%%%%%%%%%%%%%%%%%%%%%%%%%%%%%%%%%%%%%%%%%%%%%%%%%%%%%%%%%%%%%%%%%%%%%%
%
%   Section:  Introduction
%
%%%%%%%%%%%%%%%%%%%%%%%%%%%%%%%%%%%%%%%%%%%%%%%%%%%%%%%%%%%%%%%%%%%%%%%%

\section{Introduction}

In knot theory, polynomial invariants are widely used to distinguish
between different topological knots and links.
Although they are powerful tools, these invariants are
often difficult to compute: in particular, the one-variable Jones
polynomial \cite{jones85-polynomial}
and the stronger two-variable \homfly\ polynomial
\cite{freyd85-homfly,przytycki88-homfly}
are both \#P-hard to compute in general
\cite{jaeger90-jones,welsh93-complexity}.

Despite this, we can use parameterised complexity to analyse classes of
knots and links for which these polynomials become tractable to compute.
In the early 2000s, as a part of a larger work on graph polynomials,
Makowsky showed that the Jones polynomial can be computed in polynomial time
for links whose underlying 4-valent graphs have bounded
treewidth \cite{makowsky05-tutte}---in other words, the Jones polynomial
is \emph{fixed-parameter tractable} with respect to treewidth.
A slew of other parameterised
tractability results also appeared around this period for the Jones and
\homfly\ polynomials: parameters included the pathwidth of the
underlying graph \cite{makowsky03-knot}, the number of Seifert circles
\cite{makowsky03-knot,morton90-braids},
and the complexity of tangles in an algebraic presentation
\cite{makowsky03-knot}.

However, there was an important gap: it remained open
as to whether the \homfly\ polynomial is fixed-parameter tractable
with respect to treewidth.  This was dissatisfying because
the \homfly\ polynomial is both powerful and widely used, and
the treewidth parameter lends itself extremely well to building fixed-parameter
tractable algorithms, due to its strong connections to logic
\cite{courcelle87-context-free,courcelle90-rewriting} and its natural
fit with dynamic programming.

The first major contribution of this paper is to resolve this open problem: we
prove that computing the \homfly\ polynomial of a link is
fixed-parameter tractable with respect to treewidth (Theorem~\ref{t-fpt}).
Our proof gives an explicit algorithm; this is feasible to implement,
and will soon be released as part of the topological software package
{\regina} \cite{burton04-regina,regina}.

Regarding practicality: fixed-parameter tractable algorithms are
only useful if the parameter is often small, and in this sense treewidth
is a useful parameter: the underlying graph is planar, and so
the treewidth of an $n$-crossing link diagram is at worst $O(\sqrt{n})$.
This is borne out in practice---for instance, a simple greedy computation
using {\regina} shows that, for Haken's famous 141-crossing
``Gordian unknot'', the treewidth is at most 12.
Since \homfly\ is a topological invariant, one can also attempt to
use local moves on a link diagram to reduce the treewidth of
the underlying graph, and {\regina} contains facilities for this also.

% Talk about the fact that you can rewrite the knot, and that such
% facilities are implemented in {\regina}.

There are few explicit algorithms in the literature for computing
the \homfly\ polynomial in general: arguably the most notable
general algorithm is Kauffman's
exponential-time \emph{skein-template algorithm} \cite{kauffman90-state},
which forms the basis for our algorithm in this paper.
Other notable algorithms are either designed for specialised inputs (e.g.,
Murakami et~al.'s algorithm for 2-bridge links \cite{murakami14-homfly}),
or are practical but do not prove unqualified guarantees on their
complexity \cite{comoglio11-homfly,jenkins89-homfly}.

The second major contribution of this paper is to improve the worst-case
running time for computing the \homfly\ polynomial in the general case---in
particular, we prove the first \emph{sub-exponential} running time
(Corollary~\ref{c-subexp}).  Here by \emph{subexponential} we mean
$e^{o(n)}$; the specific bound that we give is $e^{O(\sqrt n \cdot \log n)}$.
The result follows immediately from analysing our fixed-parameter tractable
algorithm using the $O(\sqrt n)$ bound on the treewidth of a planar graph.

Throughout this paper we assume that the input link diagram contains no
zero-crossing components (i.e., unknotted circles that are
disjoint from the rest of the link diagram), since otherwise the number of
crossings is not enough to adequately measure the input size.
Such components are easy to handle---each zero-crossing component multiplies
the \homfly\ polynomial by $\deltaexp$, and so we simply compute the
\homfly\ polynomial without them and then adjust the result accordingly.

%%%%%%%%%%%%%%%%%%%%%%%%%%%%%%%%%%%%%%%%%%%%%%%%%%%%%%%%%%%%%%%%%%%%%%%%
%
%   Background
%
%%%%%%%%%%%%%%%%%%%%%%%%%%%%%%%%%%%%%%%%%%%%%%%%%%%%%%%%%%%%%%%%%%%%%%%%

\section{Background}

A \emph{link} is a disjoint union of piecewise linear closed
curves embedded in $\R^3$; the image of each curve is a
\emph{component} of the link.  A \emph{knot} is a link with precisely
one component.  In this paper we \emph{orient} our links by assigning a
direction to each component.

A \emph{link diagram} is a piecewise linear projection of a link onto
the plane, where the only multiple points are \emph{crossings} at which
one section of the link crosses another transversely.  The sections of
of the link diagram between crossings are called \emph{arcs}.
The number of crossings is often used as a measure of input size;
in particular, an $n$-crossing link diagram can be encoded in
$O(n \log n)$ bits without losing any topological information.

Figure~\ref{fig-links} shows two examples: the first is a knot with
4~crossings and 8~arcs, and the second is a 2-component link with
5~crossings and 10~arcs.

\begin{figure}[tb]
    \centerline{\includegraphics[scale=0.4]{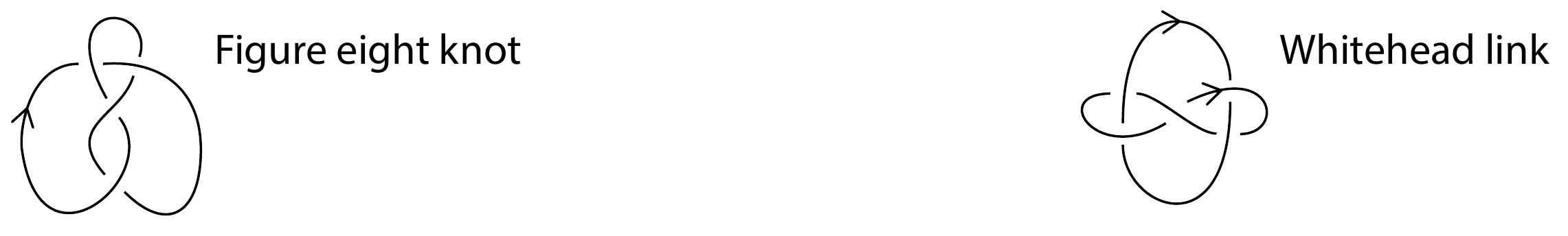}}
    \caption{Examples of knots and links}
    \label{fig-links}
\end{figure}

Each crossing has a \emph{sign} which is either
\emph{positive} or \emph{negative}, according to the direction in which
the upper strand passes over the lower; see Figure~\ref{fig-signs} for details.
The \emph{writhe} of a link diagram is the number of positive crossings
minus the number of negative crossings (so the examples from
Figure~\ref{fig-links} have writhes 0 and $-1$ respectively).

\begin{figure}[tb]
    \centerline{\includegraphics[scale=0.4]{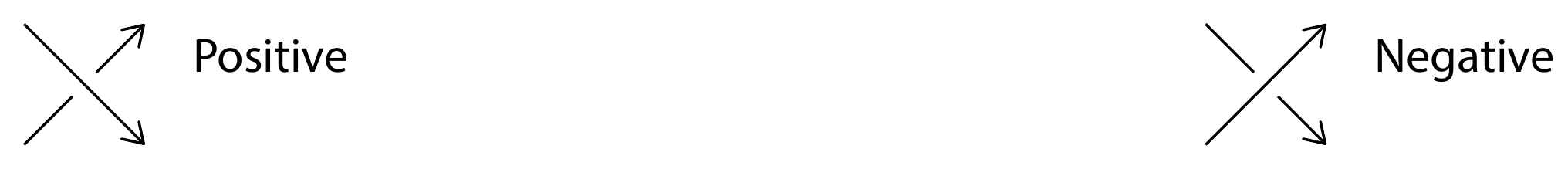}}
    \caption{Positive and negative crossings}
    \label{fig-signs}
\end{figure}

In this paper we use two operations that change a link diagram at a
single crossing.  To \emph{switch} a crossing is to move the upper
strand beneath the lower, and to \emph{splice} a crossing is to change the
connections between the incoming and outgoing arcs; see Figure~\ref{fig-ops}.

\begin{figure}[tb]
    \centerline{\includegraphics[scale=0.4]{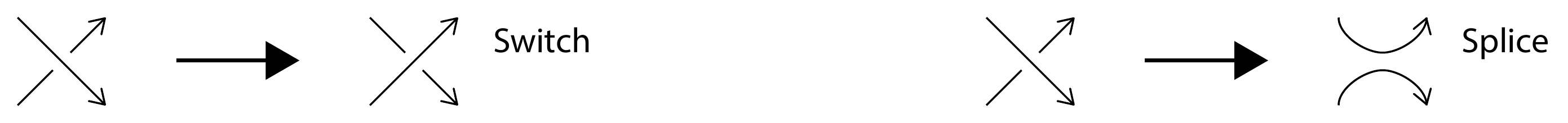}}
    \caption{Switching and splicing a crossing}
    \label{fig-ops}
\end{figure}

The \emph{\homfly\ polynomial} of a link $\link$ is a Laurent polynomial
in the two variables $\alpha$ and $z$ (a \emph{Laurent polynomial} is a
polynomial that allows both positive and negative exponents).
There are two different but essentially equivalent definitions of the
\homfly\ polynomial in the literature (the other is typically given as a
polynomial in $\ell$ and $m$ \cite{lickorish97}); we follow the same
definition used by Kauffman \cite{kauffman90-state}.

A \emph{parameterised problem} is a computational problem where the
input includes some numerical parameter $k$.  Such a problem is said to
be \emph{fixed-parameter tractable} if there is an algorithm with
running time $O(f(k) \cdot \poly(n))$, where $f$ is an arbitrary
function and $n$ is the input size.  A consequence of this is that,
for any class of inputs whose parameter $k$ is universally bounded,
the algorithm runs in polynomial time.

Treewidth is a common parameter for fixed-parameter tractable algorithms
on graphs, and we discuss it in detail now.
Throughout this paper, all graphs are allowed to be multigraphs; that
is, they may contain parallel edges and/or loops.

\begin{definition}[Treewidth] \label{d-treewidth}
Given a graph $\Gamma$ with vertex set $V$, a \emph{tree decomposition} of
$\Gamma$ consists of a tree $\tau$ and \emph{bags} $\beta_i \subseteq V$ for
each node $i$ of $T$, subject to the following constraints:
(i)~each $v \in V$ belongs to some bag $\beta_i$;
(ii)~for each edge of $\Gamma$, its two endpoints $v,w \in V$ belong to some
common bag $\beta_i$; and
(iii)~for each $v \in V$, the bags containing $v$ correspond to a
(connected) subtree of $T$.

The \emph{width} of this tree decomposition is $\max |\beta_i|-1$, and the
\emph{treewidth} of $\Gamma$ is the smallest width of any tree decomposition
of $\Gamma$.
\end{definition}

A consequence of the Lipton-Tarjan planar separator theorem
\cite{lipton79-planar}
is that every planar graph on $n$ vertices has treewidth $O(\sqrt{n})$.

\begin{definition}[Rooted tree decomposition] \label{d-rooted}
    Let $\Gamma$ be a graph.  A \emph{rooted tree decomposition} of
    $\Gamma$ is a
    tree decomposition where one bag
    is singled out as the \emph{root bag}.  We
    define children and parents in the usual way: for any adjacent bags
    $\beta,\beta'$ in the tree, if $\beta$ is closer in the tree to the root
    than $\beta'$ then we call $\beta'$ a \emph{child bag} of $\beta$,
    and we call $\beta$ the (unique) \emph{parent bag} of $\beta'$.
    A bag with no children is called a \emph{leaf bag}.

    More generally, we say that bag $\beta'$ is a \emph{descendant} of
    bag $\beta$ if $\beta \neq \beta'$ and there is some sequence
    $\beta = \beta_0, \beta_1, \beta_2, \ldots, \beta_i = \beta'$
    where each $\beta_i$ is the parent bag of $\beta_{i+1}$.
\end{definition}

\begin{definition}[Nice tree decomposition] \label{d-nice}
    Let $\Gamma$ be a graph.  A \emph{nice tree decomposition} of
    $\Gamma$ is a
    rooted tree decomposition with the following additional properties:
    \begin{enumerate}
        \item The root bag is empty.
        \item Every leaf bag contains precisely one vertex.
        \item Every non-leaf bag has either one or two child bags.
        \item If a bag $\beta_i$ has two child bags $\beta_j$ and $\beta_k$,
        then $\beta_i = \beta_j = \beta_k$; we call $\beta_i$ a
        \emph{join bag}.
        \item If a bag $\beta_i$ has only one child bag $\beta_j$, then
        either:
        \begin{itemize}
            \item $|\beta_i| = |\beta_j| + 1$ and $\beta_i \supset \beta_j$.
            Here we call $\beta_i$ an \emph{introduce bag}, and the
            single vertex in $\beta_i \backslash \beta_j$ is called the
            \emph{introduced vertex}.
            \item $|\beta_i| = |\beta_j| - 1$ and $\beta_i \subset \beta_j$.
            Here we call $\beta_i$ a \emph{forget bag}, and the single
            vertex in $\beta_j \backslash \beta_i$ is called the
            \emph{forgotten vertex}.
        \end{itemize}
    \end{enumerate}
\end{definition}

%%%%%%%%%%%%%%%%%%%%%%%%%%%%%%%%%%%%%%%%%%%%%%%%%%%%%%%%%%%%%%%%%%%%%%%%
%
%   Kauffman's algorithm
%
%%%%%%%%%%%%%%%%%%%%%%%%%%%%%%%%%%%%%%%%%%%%%%%%%%%%%%%%%%%%%%%%%%%%%%%%

\section{Kauffman's skein-template algorithm}

Kauffman's skein-template algorithm works by building a decision tree.
The leaves of this decision tree are %unlinks,
obtained from the original
link by switching and/or splicing some crossings.  Each leaf is then
evaluated as a Laurent polynomial, according to the number of components
and the specific switches and/or splices that were performed, and these
are summed to obtain the final \homfly\ polynomial.

Our fixed-parameter tractable algorithm (described in Section~\ref{s-fpt})
works by inductively constructing,
aggregating and analysing small pieces of Kauffman's decision tree.  We
therefore devote this section to describing Kauffman's algorithm in detail,
beginning with a description of the algorithm itself followed by a
detailed example.

\begin{algorithm}[Kauffman \cite{kauffman90-state}]
    Let $\link$ be a link diagram with $n$ crossings (and therefore $2n$ arcs).
    Then the following procedure computes the \homfly\ polynomial of $\link$.

    Arbitrarily label the arcs $1,2,\ldots,2n$.
    We build a decision tree by walking through the link as follows:
    \begin{itemize}
        \item Locate the lowest-numbered arc that has not yet been
        traversed, and follow the link along this arc in the direction
        of its orientation.
        \item Each time we encounter a new crossing that has not yet been
        traversed:
        \begin{itemize}
            \item If we are passing \emph{over} the crossing, then we simply
            pass through it and continue traversing the link.
            \item If we are passing \emph{under} the crossing, then we make a
            fork in the decision tree.  On one branch we \emph{splice} the
            crossing, and on the other branch we \emph{switch} the crossing.
            Either way, we then pass through the crossing (following the
            splice if we made one) and continue traversing the link.
        \end{itemize}
        \item Whenever we encounter a crossing for the second time,
        we simply pass through it (again following the splice if we made one)
        and continue traversing the link.
        \item Whenever we return to an arc that has already been traversed
        (thus closing off a component of our modified link):
        \begin{itemize}
            \item If there are still arcs remaining that have not yet been
            traversed, then we locate the lowest-numbered such arc
            % arc that has not yet been traversed
            and continue our traversal from there.
            \item If every arc has now been traversed, then the resulting
            modified link becomes a leaf of our decision tree.
        \end{itemize}
    \end{itemize}
    To each leaf of the decision tree, we assign the polynomial term
    $(-1)^{t_-} z^t \alpha^{w-w_0} \delta^{c-1}$, where:
    \begin{itemize}
        \item $t$ is the number of splices that we performed on this branch
        of the decision tree;
        \item $t_-$ is the number of splices that we performed on negative
        crossings;
        \item $w$ is the writhe of the modified link, after any switching
        and/or splicing;
        \item $w_0$ is the writhe of the original link $\link$, before any
        switching or splicing;
        \item $c$ is the number of components of the modified link;
        \item $\delta$ expands to the polynomial $\deltaexp$.
    \end{itemize}
    The \homfly\ polynomial of $\link$ is then the sum of these
    polynomial terms over all leaves. % of the decision tree.
\end{algorithm}

Note that different branches of the decision tree may traverse the arcs
of the link in a different order, since each splice changes the connections
between arcs; likewise, the modified links at the leaves of the
decision tree may have different numbers of link components.

\begin{example}
    Figure~\ref{fig-kauffman} shows the algorithm applied to the figure eight
    knot, as depicted at the root of the tree.  The eight arcs are numbered
    1--8; to help with the discussion we also label the four crossings
    $A$, $B$, $C$ and $D$, which have signs $+$, $-$, $-$ and $+$ respectively.

    \begin{figure}[tb]
        \centerline{\includegraphics[scale=0.41]{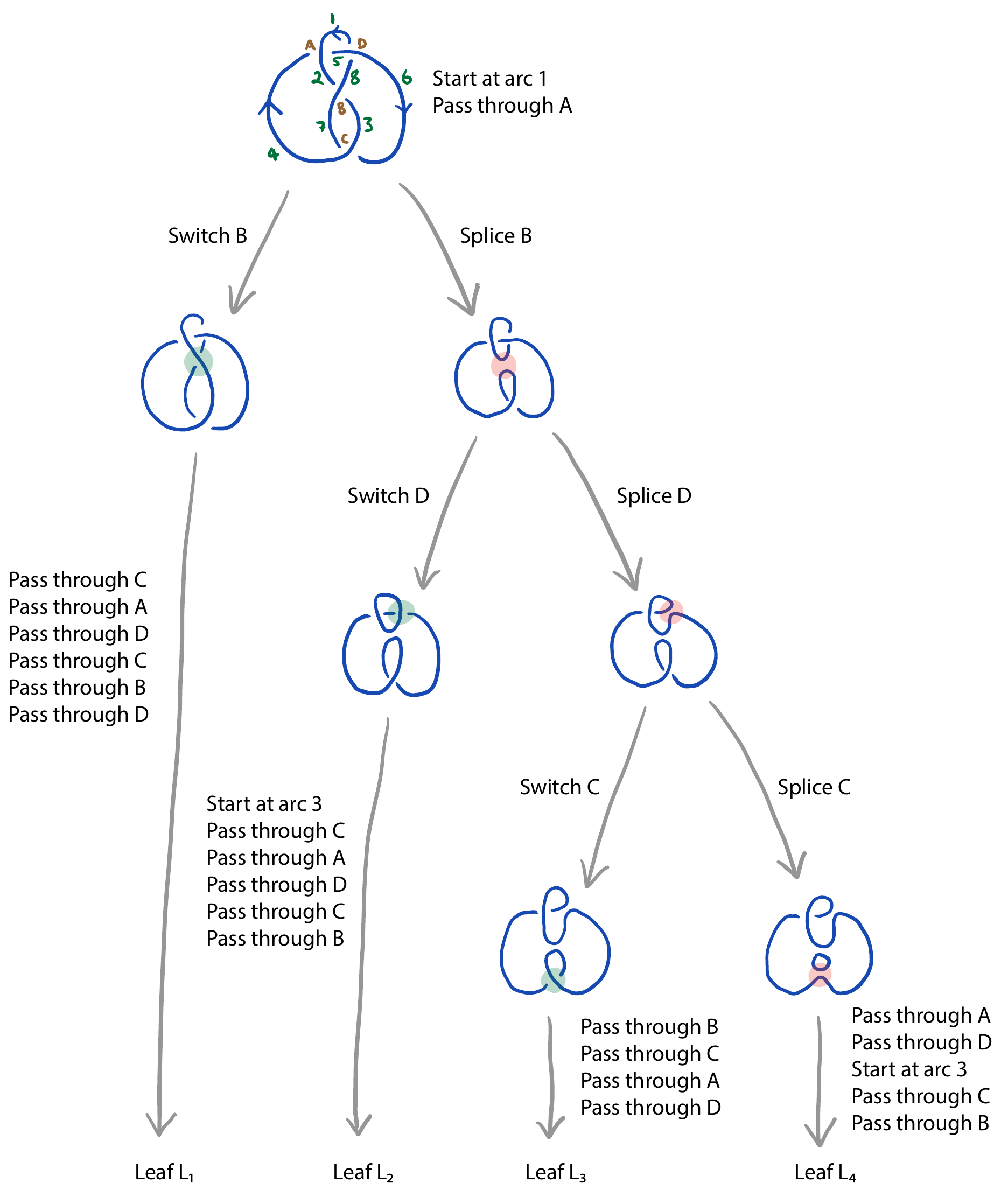}}
        \caption{Running Kauffman's skein-template algorithm}
        \label{fig-kauffman}
    \end{figure}

    We begin at arc~1, and because we first encounter crossing~$A$ on the
    upper strand, we leave it unchanged and move on to arc~2.  For
    crossing~$B$ we can either switch or splice, and in these cases the
    traversal continues to arc~3 or~8 respectively.  The decision
    process continues as shown in the diagram, resulting in the four leaves
    $L_1$, $L_2$, $L_3$ and $L_4$.

    Of particular note is the branch where we splice $B$ and then switch
    $D$.  Here the traversal runs through arcs 1, 2 and 8, at which point it
    returns to arc~1, closing off a small loop.  We now begin
    again at arc~3: this takes us through crossing $C$ (which we pass
    through because we see it first on the upper strand), then crossing
    $A$ (which we pass through because we are seeing it for the second
    time), then crossing $D$ (which we likewise pass through), and so on.

    The polynomials assigned to the four leaves are shown below:
    % note that the original knot has writhe $w_0=0$.
    \[ \small \begin{array}{l|r|r|r|r|r|c}
        & t & t_- & w & w-w_0 & c & \mbox{Poly.} \\
        \hline
        L_1 & 0 & 0 & 2 & 2 & 1 & \alpha^2 \\
        L_2 & 1 & 1 & -1 & -1 & 2 & -z\alpha^{-1}\delta \\
        \end{array} \qquad
       \begin{array}{l|r|r|r|r|r|c}
        & t & t_- & w & w-w_0 & c & \mbox{Poly.} \\
        \hline
        L_3 & 2 & 1 & 2 & 2 & 1 & -z^2\alpha^2 \\
        L_4 & 3 & 2 & 1 & 1 & 2 & z^3\alpha\delta
        \end{array}\]
    This yields the final \homfly\ polynomial
    \[ \alpha^2 - z^2\alpha^2 + (z^3\alpha - z\alpha^{-1})\delta =
    \alpha^2 - z^2\alpha^2 + (z^3\alpha - z\alpha^{-1})\deltaexp =
    \alpha^2 + \alpha^{-2} - z^2 - 1. \]
\end{example}

% XTODO: Outline why it works?  (Overcrossing-first always gives an unlink.)

\begin{theorem}
    Kauffman's skein-template algorithm computes the \homfly\ polynomial
    of an $n$-crossing link in time $2^n \cdot \poly(n)$.
\end{theorem}

\begin{proof}
    The decision tree has $\leq 2^n$ leaves, since we
    only branch the first time we traverse each crossing (and even then,
    only if we first traverse the crossing from beneath, not above).
    All other operations are polynomial time, giving an overall running
    time of $2^n \cdot \poly(n)$
\end{proof}

Although it requires exponential time, Kauffman's algorithm can compute the
\homfly\ polynomial in polynomial \emph{space}.  This is because we do not
need to store the entire
decision tree---we can simply perform a depth-first traversal through the
tree, making and undoing switches and splices as we go, and keep a
running total of the polynomial terms for those leaves that we have
encountered so far.

%%%%%%%%%%%%%%%%%%%%%%%%%%%%%%%%%%%%%%%%%%%%%%%%%%%%%%%%%%%%%%%%%%%%%%%%
%
%   FPT algorithm
%
%%%%%%%%%%%%%%%%%%%%%%%%%%%%%%%%%%%%%%%%%%%%%%%%%%%%%%%%%%%%%%%%%%%%%%%%

\section{A fixed-parameter tractable algorithm} \label{s-fpt}

In this section we present an explicit algorithm to show that computing the
\homfly\ polynomial is fixed-parameter tractable in the treewidth of the
input link diagram.

\begin{definition}
    Let $\link$ be a link diagram.  The \emph{graph of $\link$},
    denoted $\Gamma(\link)$, is the directed planar 4-valent multigraph whose
    vertices are the crossings of $\link$, and whose directed edges are the
    oriented arcs of $\link$.
\end{definition}

The first main result of this section, which resolves the open problem of
the parameterised complexity of computing the \homfly\ polynomial, is:

\begin{theorem} \label{t-fpt}
    Consider the parameterised problem whose input is a
    link diagram $\link$, whose parameter is the treewidth of the graph
    $\Gamma(\link)$, and whose output is the \homfly\ polynomial of $\link$.
    Then this problem is fixed-parameter tractable.
\end{theorem}

\subsection{Algorithm overview} \label{ss-overview}

The remainder of Section~\ref{s-fpt} is devoted to proving
Theorem~\ref{t-fpt}.  First, however, we give a brief overview of the
algorithm and the difficulties that it must overcome.

Roughly speaking, our algorithm takes a nice tree
decomposition of $\Gamma(\link)$ and works from the leaf bags to the
root bag.  For each bag of the tree, we consider a range of possible
``boundary conditions'' for how a link traversal interacts with the bag,
and for each set of boundary conditions we aggregate all ``partial leaves''
of Kauffman's decision tree that satisfy them.
We formalise these boundary conditions and their resulting aggregations
using the notions of a \emph{configuration} and \emph{evaluation}
respectively (Definitions~\ref{d-config} and \ref{d-eval}).

This general pattern of dynamic programming over a tree decomposition is
common for fixed-parameter tractable algorithms.  The main difficulty
that we must overcome is to keep the number of configurations polynomial in
the number of crossings $n$.  This difficulty arises because the
choices in Kauffman's decision tree depend upon the \emph{order} in which you
traverse the crossings, and so each configuration must encode a starting arc
for every section of a link traversal in every ``partial leaf'' of the decision
tree.
Because a ``partial leaf'' could contain $O(n)$ disjoint sections of a
traversal, each with $O(n)$ potential starting arcs,
the resulting number of configurations would grow at a rate of $O(n^n)$,
which is too large.

Our solution is the following.  Recall that Kauffman's algorithm uses an
arbitrary ordering of the arcs of the link to
determine the order in which we traverse arcs and make decisions (to pass
through, switch and/or splice crossings).  In our algorithm, we order
the arcs using the tree decomposition---for each directed arc,
we identify the forget bag in which its end crossing is forgotten, and
we then order the arcs according to how close this forget bag is to the
root of the tree.  This makes the ordering of arcs
\emph{inherent} to the tree decomposition, and so we do not need to
explicitly encode starting arcs in our configurations.  This is
enough to reduce the number of configurations at each bag to a
function of the treewidth alone, with no dependency on $n$.

\subsection{Properties of tree decompositions}

We now make some small observations about tree
decompositions of the graphs of links.

\begin{definition}
    Let $\link$ be a link diagram, let $T$ be a rooted tree decomposition
    of $\Gamma(\link)$, and let $\beta$ be any bag of $T$.  For each
    crossing $c$ of $\link$, we say that:
    \begin{itemize}
        \item $c$ is \emph{unvisited} at $\beta$
        if $c$ does not appear in either $\beta$
        or any bags in the subtree rooted at $\beta$;
        \item $c$ is \emph{current} at $\beta$
        if $c$ appears in the bag $\beta$ itself;
        \item $c$ is \emph{forgotten} at $\beta$
        if $c$ does not appear in the bag $\beta$, but
        does appear in some bag in the subtree rooted at $\beta$.
    \end{itemize}
\end{definition}

Observe that the unvisited, current and forgotten crossings at $\beta$
together form a \emph{partition} of all crossings of $\link$.

\begin{lemma} \label{l-badarcs}
    Let $\link$ be a link diagram, let $T$ be a rooted tree decomposition
    of $\Gamma(\link)$, and let $\beta$ be any bag of $T$.
    Then no arc of $\link$ can connect a crossing that is forgotten at
    $\beta$ with a crossing that is unvisited at $\beta$, or vice versa.
\end{lemma}

\begin{proof}
    Let crossing $c$ be forgotten at $\beta$, and let crossing $d$ be
    unvisited at $\beta$.  If there were an arc from $c$ to $d$ (or vice
    versa) then some bag of $T$ would need to contain both $c$ and $d$,
    by condition~(ii) of Definition~\ref{d-treewidth}.

    Since $c$ appears in a descendant bag of $\beta$ but not
    $\beta$ itself, condition~(iii) of Definition~\ref{d-treewidth}
    means that \emph{all} bags containing $c$ must be descendant bags of
    $\beta$.  However, since $d$ is unvisited, no bag containing $d$ can
    be a descendant bag of $\beta$, yielding a contradiction.
\end{proof}

\begin{lemma} \label{l-uniqueforget}
    Let $\link$ be a link diagram, let $T$ be a nice tree decomposition
    of $\Gamma(\link)$, and let $c$ be any crossing of $\link$.
    Then $T$ has a unique forget bag for which $c$ is the
    forgotten vertex.
\end{lemma}

\begin{proof}
    Since the root bag of $T$ is empty, there must be some forget bag
    for which $c$ is the forgotten vertex.  Moreover, since the bags
    containing $c$ form a subtree of $T$, there is only one bag that
    contains $c$ but whose parent does not---the root of this
    subtree.
\end{proof}

\subsection{Framework for the algorithm} \label{ss-framework}

We now define precisely the problems that we solve at each stage of the
algorithm.
Our first task is
to define a \emph{configuration}---that is, the ``boundary conditions'' that
describe how a link traversal interacts with an individual bag of the
tree decomposition.

\begin{definition} \label{d-config}
    Let $\link$ be a link diagram, let $T$ be a rooted tree decomposition
    of $\Gamma(\link)$, and let $\beta$ be any bag of $T$.  Then a
    \emph{configuration} at $\beta$ is a sequence of the form
    $(a_1, b_1, a_2, b_2, \ldots, a_u, b_u)$, where:
    \begin{enumerate}
        \item Each $a_i$ is an outgoing arc from some crossing that is
        current at $\beta$, where the destination of this arc is a
        crossing that is either current or forgotten at $\beta$.
        Moreover, every such arc must appear as exactly one of the $a_i$.
        \item Each $b_i$ is an incoming arc to some crossing that is
        current at $\beta$, where the source of this arc is a crossing
        that is either current or forgotten at $\beta$.
        Moreover, every such arc must appear as exactly one of the $b_i$.
        \item If an arc of $\link$ connects two crossings that are
        \emph{both} current at $\beta$, then by conditions~1 and~2,
        such an arc must appear as
        some $a_i$ and also as some $b_j$.  In this case we also
        require that $i=j$.
    \end{enumerate}
    We call each pair $a_i,b_i$ a \emph{matching pair} of arcs in the
    configuration, and if $a_i=b_i$ (as in condition~3 above) %XTODO
    then we call this a \emph{trivial pair}.
\end{definition}

Intuitively, each matching pair $a_i,b_i$ describes the start and end
points of a
``partial traversal'' of the link (possibly after some switches and/or splices)
that starts and ends in the bag $\beta$, and that
\emph{only passes through forgotten crossings}.  By placing these
endpoints in a sequence $a_1,b_1,\ldots,a_u,b_u$, we impose an ordering
upon these ``partial traversals'' (which we will eventually use to order the
traversal of arcs in Kauffman's decision tree).

\begin{lemma} \label{l-count}
    Let $\link$ be a link diagram, let $T$ be a rooted tree decomposition
    of $\Gamma(\link)$, and let $\beta$ be any bag of $T$.
    Then configurations at $\beta$ exist (i.e., the definition above can
    be satisfied).  Moreover, then there are at most $(2|\beta|)!^2$
    possible configurations at $\beta$.
\end{lemma}

\begin{proof}
    To show that the definition can be satisfied, all we need to show is that
    the number of arcs from a current crossing to a current-or-forgotten
    crossing (i.e., the number of arcs $a_i$) equals
    the number of arcs from a current-or-forgotten crossing to a current
    crossing (i.e., the number of arcs $b_i$).
    This follows immediately from the fact that there are
    no arcs joining a forgotten crossing with an \emph{unvisited} crossing
    (Lemma~\ref{l-badarcs}).

    The number of configuration is a simple exercise in counting: there are
    exactly $|\beta|$ crossings current at
    $\beta$, each with exactly two outgoing and two incoming arcs.
    This yields at most $2|\beta|$ arcs $a_i$ and $2|\beta|$ arcs $b_j$,
    giving at most $(2|\beta|)!$ possible orderings of the $a_i$ and
    $(2|\beta|)!$ possible orderings of the $b_i$.
\end{proof}

Our next task is to describe how we order the arcs in Kauffman's
decision tree in order to avoid having to explicitly track the start points of
link traversals in our algorithm.

\begin{definition} \label{d-ordering}
    Let $\link$ be a link diagram, and let $T$ be a nice tree decomposition
    of $\Gamma(\link)$.

    Let $a_1,\ldots,a_{2n}$ be the directed arcs of $\link$.  Let
    $c_i$ denote the crossing at the end of arc $a_i$, and let
    $\beta_i$ be the (unique) forget bag that forgets the crossing $c_i$.

    A \emph{tree-based ordering} of the arcs of $\link$
    is a total order on the arcs $\{a_i\}$ that follows a depth-first
    ordering of the corresponding bags $\{\beta_i\}$.  That is:
    \begin{enumerate}
        \item whenever $\beta_i$ is a descendant bag of $\beta_j$,
        we must have $a_j < a_i$;
        \item for any two disjoint subtrees of $T$,
        \emph{all} of the arcs whose corresponding bags appear in one
        subtree must be ordered before \emph{all} of the arcs whose
        corresponding bags appear in the other subtree.
    \end{enumerate}
\end{definition}

Essentially, this orders the arcs of $\link$ according to how close to
the root of $T$ their ends are forgotten---arcs are
ordered \emph{earlier} when the crossings they point to are forgotten
\emph{closer} to the root.
There are many such possible orderings; for our algorithm, any one will do.%
\footnote{Different tree-based orderings share many common
properties.  For example, given any collection of arcs that are
connected in $\Gamma(\link)$, all tree-based orderings share the same
\emph{minimum} arc in this collection.  This is enough to ensure that
different tree-based orderings will traverse the arcs and
crossings of $\link$ in \emph{exactly} the same order when running
Kauffman's algorithm.}

We now proceed to define an \emph{evaluation}---that is, the aggregation
that we perform for each configuration at each bag.

\begin{definition} \label{d-partialleaf}
    Let $\link$ be a link diagram and let $T$ be a nice tree decomposition
    of $\Gamma(\link)$.  Fix a tree-based ordering $<$ of the arcs of
    $\link$, and let
    $\kappa=(a_1, b_1, a_2, b_2, \ldots, a_u, b_u)$
    be a configuration at some bag $\beta$.

    A \emph{partial leaf} for $\kappa$ assigns one of the three tags
    \texttt{pass}, \texttt{switch} or \texttt{splice}
    to each forgotten crossing at $\beta$, under the following conditions.

    Consider (i)~all the forgotten \emph{crossings} of $\link$, after
    applying any switches and/or splices as described by the chosen tags;
    and (ii)~all the \emph{arcs} of $\link$ whose two endpoints are
    forgotten and/or current.  These join together to form a collection of
    (i)~connected segments of a link that start and end at current crossings
    and whose intermediate crossings are all forgotten;
    and (ii)~closed components of a link that contain only forgotten
    crossings.  We require that:

    \begin{enumerate}
        \item Each segment (as opposed to a closed component) must begin
        at some arc $a_i$ and end at the matching arc $b_i$.

        \item Suppose we traverse the segments and closed components in
        the following order.  First we traverse the segments in the order
        described by $\kappa$ (i.e., the segment from $a_1$ to $b_1$, then
        from $a_2$ to $b_2$, and so on).  Then we traverse the closed
        components according to the ordering $<$: we find the closed component
        with the smallest arc according to $<$ and traverse it starting at
        that arc; then we find the closed component with the smallest
        arc not yet traversed and traverse it from that arc; and so on.

        Then the \texttt{pass}, \texttt{switch} and \texttt{splice} tags
        that we assign to each forgotten crossing must be consistent
        with Kauffman's algorithm under this traversal.  Specifically:
        \begin{itemize}
            \item If we encounter a forgotten crossing for the first
            time on the upper strand, then it must be marked \texttt{pass}.
            \item If we encounter a forgotten crossing for the first
            time on the lower strand, then it must be marked either
            \texttt{switch} or \texttt{splice}.
        \end{itemize}
        % \item Consider each pair $a_i,b_i$.  Recall that $a_i$ is an
        % outgoing arc from some crossing that is current at $\beta$, with
        % an endpoint that is either current or forgotten.
        %
        % Follow the arc $a_i$ to its endpoint.  If the endpoint is a
        % forgotten crossing with the tag \texttt{splice} then follow the
        % splice and continue; if it is a forgotten crossing with the tag
        % \texttt{pass} or \texttt{switch} then pass through it and
        % continue.  Keep traversing arcs and forgotten crossings in this
        % way until we eventually reach a crossing that is not forgotten.
        %
        % By Lemma~\ref{l-badarcs}, we know that this eventual endpoint must
        % be a current crossing.  Our first condition on a partial leaf is that
        % this endpoint must in fact be the exact crossing $b_i$.
    \end{enumerate}
\end{definition}

This definition appears complex, but in essence, a partial leaf for $\kappa$ is
simply a choice of operations on each forgotten crossing that
\emph{could} eventually be extended to a leaf of the decision tree in
Kauffman's original algorithm.

Note that the order of traversal in condition~2 %XTODO
is indeed consistent with Kauffman's algorithm.  The segments must be traversed
before the closed components; this is because the segments will be extended and
eventually closed off as the algorithm moves towards the root of the
tree, and so the segments will eventually contain arcs that are smaller
(according to $<$) than any of the arcs in the closed components
in condition~2 above. % XTODO

\begin{definition} \label{d-eval}
    Let $\link$ be a link diagram and let $T$ be a nice tree decomposition
    of $\Gamma(\link)$.  Fix a tree-based ordering $<$ of the arcs of
    $\link$, and let $\kappa$ be a configuration at some bag $\beta$.

    The \emph{evaluation} of $\kappa$ is a Laurent polynomial in the
    variables $\alpha$, $z$ and $\delta$, obtained by summing the terms
    $(-1)^{t_-} z^t \alpha^{w-w_0} \delta^{c-1}$ over all partial leaves
    for $\kappa$, where:
    \begin{itemize}
        \item $t$ is the number of forgotten crossings marked \texttt{splice};
        \item $t_-$ is the number of forgotten negative crossings
        marked \texttt{splice};
        \item $w$ is the number of forgotten positive crossings minus
        the number of forgotten negative crossings, where we ignore any
        crossings marked \texttt{splice} and we reverse the sign of any
        crossings marked \texttt{switch};
        \item $w_0$ is the writhe of the entire original link diagram $\link$
        (including all crossings);
        \item $c$ is the number of closed components that contain only
        forgotten crossings, as described in condition~2 %XTODO
        of Definition~\ref{d-partialleaf}.
    \end{itemize}
\end{definition}

The structure of the algorithm itself is now simple: we move through the
tree decomposition from the leaf bags to the root bag, and at each bag
$\beta$ we compute the evaluation of all configurations at $\beta$.

This process eventually ends at the root bag, where we can show that
the evaluation of the
(unique) empty configuration encodes the \homfly\ polynomial of
the link $\link$:

\begin{lemma} \label{l-root-homfly}
    Let $\link$ be a link diagram and let $T$ be a nice tree decomposition
    of $\Gamma(\link)$.  Fix a tree-based ordering $<$ of the arcs of $\link$.

    Then there is only one configuration at the root bag of $T$ (which
    is the empty sequence).  Moreover, if the evaluation of this
    configuration is the Laurent polynomial $Q(\alpha, z, \delta)$, then
    the \homfly\ polynomial of $\link$ is obtained by replacing
    $\delta$ with $\deltaexp$.
\end{lemma}

\begin{proof}
    At the root bag, \emph{every} crossing is forgotten; therefore no
    crossings are current and so the only configuration is the empty sequence.
    Call this $\kappa_0$.

    Following Definition~\ref{d-partialleaf}, we then see that the
    partial leaves for $\kappa_0$ are precisely the leaves
    of the decision tree in Kauffman's skein-template algorithm,
    assuming that we order the arcs in Kauffman's algorithm using
    our tree-based ordering $<$.

    Moreover, when evaluating $\kappa_0$, the terms
    $(-1)^{t_-} z^t \alpha^{w-w_0} \delta^{c-1}$ that we sum are
    precisely the polynomials that we sum in Kauffman's algorithm, with
    the exception that we keep $\delta$ as a separate variable instead of
    expanding it to $\deltaexp$.

    It follows that, if we take this evaluation and expand
    $\delta$ to $\deltaexp$, then we obtain the same result as Kauffman's
    algorithm, which is the \homfly\ polynomial of $\link$.
\end{proof}

\subsection{Running the algorithm}

Having defined the problems to solve at each bag, we
can now describe the algorithm in full.

\begin{algorithm} \label{a-fpt}
    Suppose we are given a link diagram $\link$
    and a nice tree decomposition $T$
    of $\Gamma(\link)$.  Then the following algorithm computes the
    \homfly\ polynomial of $\link$.

    If $\link$ contains any trivial twists---that is, arcs that run from
    a crossing back to itself---then we untwist them now.
    % XTODO: Diagram?
    This preserves the topology of the link, and so does not change the
    \homfly\ polynomial.  If this produces any zero-crossing
    components then we also remove them---this \emph{does} change the
    \homfly\ polynomial (as explained in the introduction, we lose a
    factor of $\deltaexp$), but we can simply adjust the result once the
    algorithm has finished by multiplying through by $\deltaexp$ again.

    Next, fix a tree-based ordering $<$ of the arcs of $\link$.

    Now, as described at the end of Section~\ref{ss-framework},
    we work through the bags of $T$ in order from leaves
    to root.  At each bag $\beta$ we compute and store the evaluation of
    all configurations at $\beta$.  How we do this depends upon the type
    of the bag $\beta$.

        \bigskip

        \noindent \emph{If $\beta$ is a leaf bag:}

        In this case we have exactly one current crossing $c$, and every
        incoming and outgoing arc from $c$ connects it to an unvisited
        crossing.  Therefore there is only one configuration (the empty
        sequence).  Moreover, since there are no forgotten crossings at
        all, this configuration has an evaluation of
        $\alpha^{-w_0}\delta^{-1}$, where $w_0$ is the writhe of the
        entire input diagram $\link$.

        \bigskip

        \noindent \emph{If $\beta$ is an introduce bag:}

        Let $c$ be the new crossing that is introduced in $\beta$, and
        let $\beta'$ be the child bag of $\beta$.  Note that, by
        applying Lemma~\ref{l-badarcs} to the bag $\beta'$, we see that each
        of the four arcs that meets $c$ must connect $c$ to either a
        current or unvisited crossing at $\beta$.

        If all four of these arcs connect $c$ to an unvisited crossing
        at $\beta$, then the configurations at $\beta$ are precisely
        the configurations at $\beta'$.  Moreover, since the forgotten
        crossings at $\beta'$ and $\beta$ are the same, it follows that
        the partial leaves and evaluation of each configuration will be
        identical at bags $\beta'$ and $\beta$, and so we can copy all
        of our computations
        from the child bag $\beta'$ directly to $\beta$ with no changes.

        If one or more arcs connects $c$ to a current crossing at $\beta$,
        then each configuration $\kappa'$ at $\beta'$ gives rise to many
        configurations at $\beta$.  Specifically, each such arc $a$ will
        appear as a new trivial pair $a_i = b_i = a$ in the sequence
        (see condition~3 of Definition~\ref{d-config}). %XTODO
        This pair may be inserted anywhere amongst the matching pairs
        from $\kappa'$; that is, we can extend the sequence
        $(a_1, b_1, \ldots, a_u, b_u)$ to
        $(a_1, b_1, \ldots, a_j, b_j, a,a, a_{j+1}, b_{j+1}, \ldots, a_u, b_u)$
        for any insertion point $j$.
        As before, the partial leaves after this insertion are identical to the
        partial leaves for $\kappa'$, and so the evaluation of
        each such new configuration is identical to the evaluation of
        $\kappa'$.

        \bigskip
        
        \noindent \emph{If $\beta$ is a join bag:}

        Let $\beta_1$ and $\beta_2$ be the two child bags of $\beta$.
        We iterate through all pairs $(\kappa_1,\kappa_2)$ where
        each $\kappa_i$ is a configuration at $\beta_i$, and attempt
        to find ``compatible'' pairs that can be merged to form a
        configuration at $\beta$.
        Note that all forgotten crossings
        at $\beta_1$ will be unvisited at $\beta_2$, and
        all forgotten crossings at $\beta_2$ will be
        unvisited at $\beta_1$.

        The only arcs that appear in the sequences for \emph{both}
        $\kappa_1$ and $\kappa_2$ are those arcs whose endpoints are both
        current at $\beta$.  By Definition~\ref{d-config}, such arcs
        must appear as trivial pairs in both $\kappa_1$ and $\kappa_2$.
        Therefore, if these trivial pairs all appear in the same relative order
        in both $\kappa_1$ and $\kappa_2$, we can merge $\kappa_1$ and
        $\kappa_2$ to form a configuration at $\beta$---in fact there
        are many ways to do this, since we can interleave the two
        sequences for $\kappa_1$ and $\kappa_2$ however we like as long as
        the matching pairs from each individual $\kappa_i$ are all kept in the
        same relative order.
        
        Since the forgotten
        crossings for $\beta_1$ and $\beta_2$ are disjoint, we can combine any
        partial leaf for $\kappa_1$ with any partial leaf
        for $\kappa_2$ to form a partial leaf for the new configuration
        $\kappa$ at $\beta$. Therefore the
        evaluation of $\kappa$ is $e_1 \cdot e_2 \cdot \alpha^{w_0}\delta$,
        where each $e_i$ is the evaluation of $\kappa_i$.
        Here the extra factor of $\alpha^{w_0}\delta$ compensates for the
        fact that each polynomial term from Definition~\ref{d-eval} includes a
        ``constant factor'' of $\alpha^{-w_0}\delta^{-1}$ which we
        inherit twice from $e_1$ and $e_2$.

        If the trivial pairs for $\kappa_1$ and $\kappa_2$ do \emph{not}
        appear in the same relative order in both sequences,
        then we cannot merge the
        two configurations to form a new configuration at $\beta$, and
        so we ignore this pair of configurations $(\kappa_1,\kappa_2)$
        and move on to the next pair.

        \bigskip
        
        \noindent \emph{If $\beta$ is a forget bag:}

        Let $c$ be the crossing that is forgotten in $\beta$, and
        let $\beta'$ be the child bag of $\beta$.
        Once more we iterate through all configurations at
        $\beta'$; let $\kappa'$ be such a configuration.

        We consider applying each of the tags \texttt{pass},
        \texttt{switch} and \texttt{splice} to the forgotten crossing $c$.
        For consistency with Kauffman's decision tree,
        we only allow the \texttt{pass} tag if
        the upper incoming arc into $c$ appears
        earlier in $\kappa'$ than the lower incoming arc into $c$ (which means
        we first encounter $c$ on the upper strand); likewise, we
        only allow the \texttt{switch} and \texttt{splice} tags if the
        lower incoming arc into $c$ appears earlier in $\kappa'$ than the
        upper incoming arc into $c$.

        Having chosen a tag, we then attempt to
        convert $\kappa'$ into a new configuration $\kappa$ at $\beta$.
        This involves combining
        matching pairs of $\kappa'$ that connect with $c$
        to reflect how the link traversal passes through
        the now-forgotten crossing $c$.
        There are two ways that this can be done:
        \begin{itemize}
            \item Matching pairs on either side of $c$ could be adjacent in
            $\kappa'$.  For instance, suppose we apply the \texttt{switch} tag.
            Then $\kappa'$ could be of the form
            $\ldots, a_i, b_i, \allowbreak a_{i+1}, b_{i+1}, \ldots$, where
            $b_i$ is an incoming arc for $c$ and $a_{i+1}$ is the
            opposite outgoing arc for $c$ (in the case of \texttt{splice},
            $a_{i+1}$ would need to be the \emph{adjacent} outgoing arc
            instead).
            The new configuration $\kappa$ would then be
            $\ldots, a_i, b_{i+1}, \ldots$; here
            $(a_i,b_{i+1})$ becomes a new matching pair.

            \item There could be a single matching pair in $\kappa'$ that runs
            from $c$ back around to itself.  For instance, if we apply
            the \texttt{switch} tag then $\kappa'$ could be of the form
            $\ldots, a_i, b_i, \ldots$, where
            $a_i$ is an outgoing arc for $c$ and $b_i$ is the
            opposite incoming arc (again, for \texttt{splice}
            we would need $b_i$ to be the \emph{adjacent} incoming arc
            instead).
            In this case, forgetting $c$ will connect the two ends of the
            traversal segment from $a_i$ to $b_i$ to form
            a new closed link component, and the new configuration $\kappa$
            is obtained by deleting the pair $(a_i,b_i)$ from $\kappa'$.
        \end{itemize}
        Note that we must combine \emph{two} pairs of matching pairs---one
        for each strand that passes through $c$.  If this cannot
        be done as described above (i.e., the relevant matching pairs
        are neither adjacent in $\kappa'$ nor do they run from $c$ back
        to itself), then we cannot apply our chosen tag to
        $\kappa'$.  In this case we just move to the next choice of tag
        for $c$ and/or the next available configuration at $\beta'$.

        If we \emph{are} able to use our chosen tag with $\kappa'$, then
        we can use the evaluation of $\kappa'$ to compute the evaluation of
        the new configuration $\kappa$.  We must, however, multiply by
        an appropriate factor to reflect how the partial leaves have
        changed, following Definition~\ref{d-eval}:
        \begin{itemize}
            \item if we chose \texttt{splice} then we must multiply
            by $z$, and also by $-1$ if $c$ is a negative crossing;
            \item if we \texttt{pass} a positive crossing or
            \texttt{switch} a negative crossing, we must multiply by
            $\alpha$;
            \item if we \texttt{pass} a negative crossing or
            \texttt{switch} a positive crossing, we must multiply by
            $\alpha^{-1}$;
            \item if we formed a new closed link component then we must
            multiply by $\delta$.
        \end{itemize}

        Since $\kappa$ is obtained by deleting and/or merging matching pairs
        from $\kappa'$, it is possible that several different child
        configurations $\kappa'$ could yield the \emph{same} new
        configuration $\kappa$.  If this happens, we simply
        sum all of the resulting evaluations at $\kappa$.

    \bigskip

    Once we have finished working through all the bags, we take the evaluation
    of the unique configuration at the root bag and expand
    $\delta$ to $\deltaexp$ as described in Lemma~\ref{l-root-homfly}.
    This yields the final \homfly\ polynomial of $\link$.
\end{algorithm}

\begin{theorem} \label{t-maintime}
    If the nice tree decomposition in Algorithm~\ref{a-fpt} has
    $O(n)$ bags and width $k$, then the algorithm has running time
    $O\left((2k)!^4 \cdot \poly(n)\right)$,
    where $n$ is the number of crossings in the link diagram $\link$.
\end{theorem}

\begin{proof}
    Most of the operations in the algorithm are clearly polynomial time,
    and we do not discuss their precise complexities here.
    The only source of super-polynomial running time comes from the
    large number of
    configurations to process at each bag.

    When processing a forget or introduce bag, Lemma~\ref{l-count} shows
    that there are $\leq (2k)!^2$ child configurations to process,
    requiring $O\left((2k)!^2 \cdot \poly(n)\right)$ time in total.
    When processing a join bag, we iterate through \emph{pairs} of
    configurations $(\kappa_1,\kappa_2)$, and so
    the total processing time becomes
    $O\left((2k)!^4 \cdot \poly(n)\right)$.
    Note that, although any individual pair $(\kappa_1,\kappa_2)$ could yield a
    super-polynomial number of new configurations $\kappa$ (due to
    the many possible ways to merge configurations), these 
    nevertheless contribute
    to a \emph{total} number of configurations at the join bag which
    is still bounded by Lemma~\ref{l-count}, and so the total processing
    time at a join bag remains no worse than
    $O\left((2k)!^4 \cdot \poly(n)\right)$.
\end{proof}

Algorithm~\ref{a-fpt} assumes that you already have a tree decomposition;
however, finding one with the smallest possible width is an NP-hard
problem \cite{cygan15-fpt}.  We therefore extend our running time
analysis to the more common case where only the link is given, and a
tree decomposition is \emph{not} known in advance.

\begin{corollary} \label{c-time}
    Given a link diagram $\link$ with $n$ crossings whose graph
    $\Gamma(\link)$ has treewidth $k$, it is possible to compute the
    \homfly\ polynomial of $\link$ in
    time $O\left((8k)!^4 \cdot \poly(n)\right)$.
\end{corollary}

\begin{proof}
    Cygan et~al.\ \cite{cygan15-fpt} give an algorithm that can construct
    a tree decomposition with width $\leq 4k+4$ and $n$ bags in time
    $O(8^k k^2 \cdot n^2)$.  Kloks \cite{kloks94-treewidth} then shows how
    to convert this into a \emph{nice} tree decomposition in $O(n)$ time
    with the same width, and with $O(n)$ bags.
    Our corollary now follows by applying Theorem~\ref{t-maintime} with
    width $4k+4$.
    Note that the running time from
    Theorem~\ref{t-maintime} dominates the preprocessing time required to
    build the tree decompositions.
\end{proof}

Corollary~\ref{c-time} shows that computing the \homfly\ polynomial is
fixed-para\-meter tractable, thereby finally
concluding the proof of Theorem~\ref{t-fpt}, our first main result.

However, unlike Kauffman's algorithm, our algorithm is \emph{not} polynomial
space, since it must store up to $(2k)!^2$ configurations and their
evaluations at each bag.

We can now finish this paper with our second main result.
Since the treewidth of a planar
graph is $O(\sqrt n)$, we can substitute $k = O(\sqrt n)$ into
Corollary~\ref{c-time} to yield the following:

\begin{corollary} \label{c-subexp}
    Given a link diagram $\link$ with $n$ crossings,
    it is possible to compute the \homfly\ polynomial of $\link$ in
    time $e^{O(\sqrt n \cdot \log n)}$.

    That is, it is possible to compute the \homfly\ polynomial in
    sub-exponential time.
\end{corollary}

%%%%%%%%%%%%%%%%%%%%%%%%%%%%%%%%%%%%%%%%%%%%%%%%%%%%%%%%%%%%%%%%%%%%%%%%
%
%   Bibliography
%
%%%%%%%%%%%%%%%%%%%%%%%%%%%%%%%%%%%%%%%%%%%%%%%%%%%%%%%%%%%%%%%%%%%%%%%%

% XTODO: Check arXiv references to see if they have since been published.
\bibliographystyle{amsplain}
\bibliography{pure}

\end{document}